\documentclass{article}
\usepackage[utf8]{inputenc}
\usepackage[T1]{fontenc}
\usepackage{amsmath,amssymb,amsthm,bbm,float,graphicx,geometry,lmodern,mathtools,parskip,setspace,subcaption}
\usepackage{mathrsfs}
\usepackage{enumerate}
\usepackage{nicefrac}
\usepackage{todonotes}
\usepackage{comment}
\usepackage[colorlinks, linkcolor = blue!80!black, citecolor = blue!80!black, breaklinks, pdfauthor={Stein Andreas Bethuelsen, Christian Moench}]{hyperref}
\usepackage{orcidlink}
\usepackage{titling}
\usepackage{fullpage}

\usepackage[varg]{pxfonts}

\usepackage{tikz}
\usetikzlibrary{backgrounds}
\usetikzlibrary{patterns}
\usetikzlibrary{positioning, shapes.geometric}

\usepackage{caption}
\usepackage{graphicx}
\graphicspath{{Images/}}
\allowdisplaybreaks

\newcommand{\N}{\mathbb{N}}

\newtheorem{theorem}{Theorem}[section]
\newtheorem{lemma}[theorem]{Lemma}

\newtheorem{prop}[theorem]{Proposition}

\theoremstyle{definition}

\newtheorem{remark}[theorem]{Remark}

\usepackage{nameref}

\makeatletter
\let\orgdescriptionlabel\descriptionlabel
\renewcommand*{\descriptionlabel}[1]{%
  \let\orglabel\label
  \let\label\@gobble
  \phantomsection
  \edef\@currentlabel{#1}%
  \let\label\orglabel
  \orgdescriptionlabel{#1}%
}

\renewcommand{\P}{\mathbb{P}}

\newcommand{\1}{\mathbbm{1}}

\newcommand{\E}{\mathbb{E}}

\newcommand{\Z}{\mathbb{Z}}

\pretitle{\centering\LARGE\scshape}
 \posttitle{\vskip 0.75cm}

 \predate{\vskip 0.5 cm \centering\large}
 \postdate{\par}

\usepackage[style = numeric, sorting=nyt, url = false, abbreviate=false, maxbibnames=9, sortcites=true, doi = true, backend = biber, giveninits = true, isbn=false]{biblatex}
\renewbibmacro{in:}{\ifentrytype{article}{}{\printtext{\bibstring{in}\intitlepunct}}}
\AtEveryBibitem{\clearlist{language}}

\title{Strict monotonicity of critical points in independent long-range percolation models}

\thanksmarkseries{arabic}

\author{
Stein Andreas Bethuelsen \thanks{University of Bergen, All\'egaten 41,
5020 Bergen, Norway} \\ stein.bethuelsen@uib.no
\and
Christian M\"{o}nch \orcidlink{0000-0002-6531-6482}
\\ cmoench25@gmail.com
}

\date{\today}

\begin{document}
\maketitle

\begin{spacing}{0.9}
\begin{abstract}
\noindent
We consider independent long-range percolation models on locally finite vertex-transitive graphs. Using coupling ideas we prove strict monotonicity of the critical points with respect to local perturbations in the connection function, thereby improving upon previous results obtained via the classical essential enhancement method of Aizenman and Grimmett in several ways. In particular, our approach allows us to work under minimal assumptions, namely shift-invariance and summability of the connection function, and it applies to both undirected and directed bond percolation models.

\smallskip
\noindent\footnotesize{{\textbf{AMS-MSC 2020}: Primary 60K35; Secondary 82B43.}

}
\end{abstract}
\end{spacing}

\section{Background and motivation}
 
Consider i.i.d.\ nearest neighbour bond percolation on the integer lattice $\Z^d$. A classical sensitivity result of Aizenman and Grimmett \cite{AizGrim91}, cf.\ \cite{balister2014essentialenhancementsrevisited}, states that the critical percolation threshold $p_\mathsf{c}(\Z^d)$ is \emph{strictly} decreasing in the dimension $d$. Generalising earlier results of Kesten \cite{KestenBook} and Menshikov \cite{Menshikov87}, the work \cite{AizGrim91} went far beyond Bernoulli percolation on $\Z^d$. This is achieved through the general notion of \emph{essential enhancements} to obtain strict monotonicity results for critical thresholds under \emph{local perturbations} in percolation models. The crucial technical ingredient for this theory are differential inequalities obtained through Margulis--Russo-type formulas. The method of \cite{AizGrim91} remains the standard approach to this type of question; see \cite{Martineau19,taggi23} for some recent applications and \cite{PenRoso11} for a continuum version of the argument for bounded range Poisson random-connection models.

Here, we investigate \emph{long-range} percolation models. In their most straightforward form, they consist of a random graph $G$ with vertex set $\Z^d$ in which edges are generated independently with respect to some translation invariant rule, i.e.
\[
\P(\{x,y\} \in E(G))=J(x-y), \quad x,y\in\Z^d,
\]
for some \emph{connectivity function} $J:\Z^d\to[0,1]$ with $J(x)=J(-x)$ and $\sum_{x\in\Z^d}J(x)<\infty.$ We may then ask the following variant of the question of strict monotonicity: 

\textit{Suppose that $J,J'$ are connectivities such that the corresponding graphs $G,G'$ contain an infinite connected component almost surely. If $J'<J$ in the coordinate-wise sense, is it true that $p_\mathsf{c}(G)<p_\mathsf{c}(G')$?} 

As far as we know, this problem has not been comprehensively addressed, neither in a discrete nor in a continuum setting, and the few results known for $J$ with unbounded support \cite{bäumler2025continuitycriticalvalueshape,rosoman2011critical} all rely on the tools of \cite{AizGrim91}. However, the differential inequality approach fails to yield optimal results in this setting since it necessarily leads to Lipschitz-type conditions on the connectivity functions involved. This is an artefact of the technique, which requires the perturbation of the model to be `continuously spreadable' in a certain sense. When adapting the approach to long-range models, the effect of the perturbation therefore needs to be localisable in a controlled way, which produces additional regularity requirements on the connectivity function.

Motivated by these shortcomings of the Aizenman--Grimmett approach in the long-range setting, we develop a new argument to obtain strict inequality of critical values, which relies on stochastic domination techniques instead of differential inequalities. Our work is broadly inspired by the papers \cite{duminil-copin_new_2016,vanneuville2023sharpness,vanneuville2024exponential}, and has a somewhat similar flavour to the techniques used for several (not necessarily independent) nearest-neighbour or finite range models in \cite{bäumler2025localcriteriaglobalconnectivity,Mühlbacher21,klippel25,LimaUngarettiVares2024,martineau2025stochasticdominationliftsrandom}.

\section{Main results}\label{sec:percolation}

We now state the problem and our main results for long-range percolation rigorously in a more general setting as the one discussed above.

\paragraph{Long-range percolation on transitive graphs.}

Following \cite{duminil-copin_new_2016}, let $\Gamma=(V(\Gamma),E(\Gamma))$ denote a locally finite vertex-transitive graph with a distinguished origin vertex $o\in \Gamma$. Throughout, we use $x\in\Gamma$ instead of $x\in V(\Gamma)$, because the edge set of $\Gamma$ will appear only indirectly. Instead, we focus on the set $\Gamma^{[2]}=\{xy: x,y\in \Gamma, x\neq y\}$ of \emph{potential edges}. We often use the notation $e\in\Gamma^{[2]}$ to denote a generic edge instead of $xy\in \Gamma^{[2]}$, if we do not want to specify the endpoints. Note that $xy$ is used as a shorthand for the more cumbersome $\{x,y\}$, and we mainly work with undirected graphs.

We consider independent Bernoulli configurations $\omega$ on $\Gamma^{[2]}$ satisfying
\[
\P(\omega_{e}=1)=J_{e}, \quad e\in\Gamma^{[2]},
\]
for a {connectivity function} $J:\Gamma^{[2]}\to [0,1]$. The connectivity function is adapted to the graph structure of $\Gamma$ in the following way: we assume that there exists a group $\mathsf{S}\subset\operatorname{aut}(\Gamma)$ of automorphisms of $\Gamma$ that acts transitively on (the vertices of) $\Gamma$ such that $J_{s(x)s(y)}=J_{xy}$ for all $s\in\mathsf{S}$. We call such connectivity functions $\mathsf{S}$-\emph{invariant}. For fixed $\mathsf{S}$, we denote by $\mathscr{J}=\mathscr{J}(\Gamma,\mathsf{S})$ the family of all $\mathsf{S}$-invariant connectivity functions satisfying $\sum_{x\in\Gamma}J_{ox}<\infty$. If $J\in\mathscr{J}(\Gamma,\mathsf{S})$ for some $\mathsf{S}$, we call $J$ simply \emph{summable and invariant}.

We use the natural component-wise partial order on connection functions with the usual notational convention that
\[
J'<J \text{ if }J'_{e}\leq J_{e} \text{ for all }e\in\Gamma^{[2]} \text{ and }J-J' \text{ is not identically }0.
\]
A random configuration $\omega:\Gamma^{[2]}\to \{0,1\}$ corresponds to a random subgraph $G(\omega)$ of $(\Gamma,\Gamma^{[2]})$ in the obvious way. We will generally work with these random subgraphs and write
$G_J$ for a realisation of the long-range percolation model on $\Gamma$ with $J\in \mathscr{J}(\Gamma,\mathsf{S})$, where we usually suppress the dependence on the Bernoulli configuration $\omega$ in the notation.

\begin{remark}
The Borel--Cantelli Lemma readily implies that $\sum_{x\in\Gamma}J_{ox}<\infty$ is equivalent to almost sure local finiteness of $G_J.$ Hence, if $J$ and $J'$ both are non-summable, there is no phase transition. On the other hand, if $\sum_{x\in\Gamma}J'_{ox}<\sum_{x\in\Gamma}J_{ox}=\infty$, then $p_\mathsf{c}(G_{J'})>0=p_\mathsf{c}(G_{J})$ by a simple branching process comparison. Hence, it suffices to study summable connection functions.
\end{remark}
For $p\in(0,1)$, we write $pJ$ for the connection function $\{pJ_e,e\in\Gamma^{[2]}\}$. It is elementary to see that $G_{pJ}$ has the same distribution as an i.i.d.\ Bernoulli bond percolation model with retention probability $p$ on $G_J$. We say that \emph{percolation occurs} for $J$, if 
\[
\P\left(\left|\left\{x\in \Gamma: o\overset{G_{J}}{\leftrightarrow}x\right\}\right|=\infty\right)>0,
\]
where $\{o\overset{G_{J}}{\leftrightarrow}x\}$ denotes the event that $o$ is connected to $x$ within $G_J$.  
We furthermore say that 
\begin{itemize}
    \item $J$ is \emph{critical} if, for any choice of $\varepsilon>0$, percolation occurs for $(1+\varepsilon)J\wedge 1$ and percolation does not occur for $(1-\varepsilon)J$,
    \item $J$ is \emph{subcritical} if there exists $\varepsilon>0$, such that percolation does not occur for $(1+\varepsilon)J\wedge 1$, and
    \item $J$ is \emph{supercritical} if there exists $\varepsilon>0$, such that percolation does occur for $(1-\varepsilon)J$.
\end{itemize}
 The above definitions are characterized by how the event that percolation occurs is affected by a global perturbation of the connection function and are guided by the comparison with i.i.d.\ nearest neighbour bond percolation, see also the discussion of critical behaviour in \cite{Berger2002}. 
 
 \paragraph{Main results.} 
 The goal of this paper is to establish that criticality is sensitive to local perturbations of the connectivity function.  
 To formalize this, denote by $\mathscr{J}_{<1}(\Gamma,\mathsf{S})\subset\mathscr{J}(\Gamma,\mathsf{S})$ the summable and invariant connectivity functions which do not assume the value $1$. Note that $J\in\mathscr{J}(\Gamma,\mathsf{S})$ can be subcritical only if it belongs to this more restrictive class. We say that $J\in\mathscr{J}(\Gamma,\mathsf{S})$ is \emph{strongly critical}, if 
\begin{itemize}
    \item for any $J'\in \mathscr{J}(\Gamma,\mathsf{S})$ with $J'<J$ it holds that $\E\left[\left|\left\{x\in \Gamma: o\overset{G_{J'}}{\leftrightarrow}x\right\}\right|\right]<\infty$.
    \item and percolation occurs for every $J''\in \mathscr{J}(\Gamma,\mathsf{S})$ with $J''>J$.
\end{itemize}

\begin{theorem}[{Characterisation of critical parameter set}]\label{thm:main}
Let $J\in\mathscr{J}_{<1}(\Gamma,\mathsf{S})$. 
Then $J$ is strongly critical if and only if $J$ is critical. 
\end{theorem}

Our proof relies on a coupling of the origin cluster of $G_{J'}$ to a slightly perturbed version of the origin cluster in $G_{pJ}$ for some $p=p(J,J')<1$. The coupling is based on a local exploration scheme of the clusters containing the origin. For the percolative phase, this implies domination of a critical connection function as a sufficient condition for supercriticality.
\begin{theorem}[{Well-behaviour under upward perturbation}]\label{thm:main3}
Let $J \in \mathscr{J}(\Gamma,\mathsf{S})$. If there exists a critical $J' \in \mathscr{J}_{<1}(\Gamma,\mathsf{S})$ with $J>J'$, then $J$ is supercritical.
\end{theorem}

Similarly, our method leads to a generalised variant of the subcritical sharpness results of \cite{AB87} in the following way:
\begin{theorem}[{Well-behaviour under downward perturbation}]\label{thm:main2}
	Let $J \in \mathscr{J}_{<1}(\Gamma,\mathsf{S})$. If there exists a critical $J'' \in \mathscr{J}(\Gamma,\mathsf{S})$ with $J<J''$, then $J$ is subcritical and 
	\[
	\E\left[\left|\left\{x\in \Gamma: o\overset{G_{J}}{\leftrightarrow}x\right\}\right|\right]<\infty.
	\]
	If $J$ is in addition finitely supported, then \[ \P\left(o\overset{G_{J}}{\leftrightarrow}B_\Gamma(o,n)^\mathsf{c}\right)\le \textup{e}^{-c(\Gamma,J)n},\text{ for all }n\in\N,\]
	where $c(\Gamma,J)>0$ is a model-dependent constant and $B_\Gamma(o,n)$ denotes the ball of radius $n$ around $o$ in $\Gamma$.
\end{theorem}

\paragraph{Organisation of the manuscript.} The following section is devoted to the proofs of our main results. 
We first give a heuristic explanation of our proof, before detailing our coupling argument involving a local exploration algorithm of the percolation cluster. Then we state and prove Proposition \ref{prop:coupling}, which is our key technical result which provides the aforementioned stochastic domination, and from which we derive our main results. 
Section~\ref{sec:conclusion} contains a further discussion of how our results extend beyond the above setting, e.g.\ to directed and oriented percolation, and of related recent works. 

\section{Proof of main result}\label{sec:proof}

\subsection{Heuristic explanation of the proof}\label{sec:heuristic}
Let us provide some intuition as to how our argument works. Given $J'<J$, we may couple the associated percolation models $G=G_J,G'=G_{J'}$ in the obvious way to obtain $G'\subset G$ under the coupling. In particular, if $\Delta=\textrm{supp}(J-J')$ denotes the set of coordinates on which $J'$ differs from $J$, we may view $G'$ as an \emph{independent inhomogeneous} percolation of $G$. Let us call an edge $e\in E(G)$ \emph{fragile} if \emph{all} its adjacent edges in $G$ are in $\mathsf{S}$-translates of $\Delta$. Thus, with a small probability $\varepsilon$, $e$ is isolated in $G'$. We say $e$ is \emph{shattered} if this occurs, since it is of no use for achieving percolation in $G'$. Note that fragility is determined by the neighbourhood of the edge in $G$, whereas the property of being shattered is determined by the neighbourhood in $G'$. By invariance of $J$ and $J'$, each edge $e$ has the same probability of being shattered, however the fragility and shattering status of edges is not independent. We have thus related the \emph{independent inhomogeneous} percolation of $G$ to a \emph{dependent homogeneous} percolation. Nevertheless, the dependencies are not very complicated: edges of $G$ that do not have an adjacent edge in common obtain their status independently. 

If $J$ has bounded support, one may now directly apply the classical domination result of Liggett, Schonman and Stacey \cite[Theorem 1.3]{LSS97} to conclude that the shattered edges dominate an i.i.d.\ field of intensity $\varepsilon'\ll \varepsilon$ over $E(G)$. In other words, the downward-effect of going from $G$ to $G'$ is at least as strong as performing an independent bond percolation on $G$ with retention parameter $1-\varepsilon'$, which implies subcriticality of $J'$, if $J$ is sufficiently close to critical. Since we work with unbounded connectivity functions, this approach does not quite work, but it nonetheless provides a good intuition for what our algorithmic construction in the following section is designed to achieve. We replace the \emph{global} domination by a product measure of \cite{LSS97} by a \emph{local} coupling of the cluster exploration, which is flexible enough to also work in the infinite support setting.

\subsection{Exploration algorithm}\label{sec:algo} We now describe the exploration algorithm that is at the heart of our coupling arguments. The parameters of the algorithm are 
\begin{itemize}
\item the connectivity function $J$, 
\item A finite set $\Delta \subset V$,
\item some small number $q>0$,
\item and an integer $n>0$.
\end{itemize}
Later, in the proofs, we will set $\Delta\subset\textup{supp}(J-J')$, and the number $q$ will be carefully chosen for our purpose as a function of $J$ and $J'$. The parameter $n$ is there to ensure that we limit our exploration to the edges in $B(o,n)$ so that our algorithm terminates after finitely many steps\footnote{This is a mere matter of taste -- it is not difficult to describe a variant of the algorithm that explores a (potentially) infinite cluster.}. 

We work on an extended probability space $\bar{\Omega}$ that carries
\begin{itemize}
    \item a field of \emph{edge marks}, i.i.d.\ $\operatorname{Uniform}(0,1)$ random variables $\{U_e, e\in \Gamma^{[2]}\}$ to sample the edges of $G_J$,
    \item for each edge $e=xy\in \Gamma^{[2]}$ an independent triplet $(V_{xy}^x, W_{xy}, V_{xy}^y)$ of i.i.d.\ $\operatorname{Uniform}(0,1)$ random variables to perform various additional percolation and randomisation steps. 
    We call these the \emph{auxiliary edge marks} and they are used to describe how we explore edges of $G_{J'}$ in $G_J$.
\end{itemize}

The algorithm explores the random configuration locally around $o$ by adding edges that are present in a percolated subgraph $H_n$ of $G=G_J$ together with a collection of edges in $H_n$ that are \emph{tagged}. These tagged edges are later going to be coupled to leaves of the cluster of $o$ in $G_{J'}$. 

Let $\mathcal{E}[n]$ denote the set of all potential edges with both endpoints 
in $B_\Gamma(o,n) \coloneqq \{x \in V \colon d_{\Gamma}(o,x)\leq n\}$ and let 
\[ \mathcal{T}=\mathcal{T}(G,n)\coloneqq \big\{v \in B_\Gamma(o,n):\textup{dist}_G\big(v,B_\Gamma(o,n)^\mathsf{c}\big)=1 \big\},\]
 where $d_{\Gamma}$ is the graph distance on $\Gamma$ and $\textup{dist}_G$ denotes the induced graph distance on $G$. 
At the initialisation of the algorithm, all edge marks $U_{e}$ with $e\in\Gamma^{[2]}\setminus\mathcal{E}[n]$ (and therefore the set $\mathcal{T}(G,n)$) are known. The algorithm iteratively reveals certain edge marks (and auxiliary edge marks) assigned to edges in $\mathcal{E}[n]$, starting from the origin.

As the algorithm reveals more and more vertices in $G$, it eventually terminates either upon running out of viable edges to process or by establishing a path from the origin to $\mathcal{T}(G,n)$, which implies that $H_n$ locally percolates.

The algorithm operates using the following lists for each integer $t\geq 0$:
\begin{itemize}
    \item $A_t \subset V$ - active vertices after exploration stage $t$.
    \item $B_t \subset V$ - boundary vertices after exploration stage $t$.
    \item $E_t \subset \mathcal{E}[n]\times(0,1)$ - all unexplored edges $e$ for which mark information has been revealed up to and including exploration stage $t$, together with their respective mark $U_e$.
    \item $L_t \subset \mathcal{E}[n]$ - unexplored edges after exploration stage $t$.
\end{itemize}
Each exploration stage involves the exploration of a single edge. An exploration stage may include an \emph{(F)-check} or an \emph{(S)-check}, the procedures of which are explained below. They represent adapted versions of fragile and shattered edges, respectively, as they appeared in the heuristics of Section~\ref{sec:heuristic}. During these checks (and only there), mark information of unexplored edges is potentially revealed, which potentially creates dependencies between the marks revealed in the exploration steps.

We initialise the algorithm by setting 
\begin{equation}
    A_0=\{0\}, \quad 
B_0=\emptyset, \quad
  E_0=\emptyset, \quad 
    L_0=\mathcal{E}[n],
\end{equation}
together with a uniform random ordering of $\mathcal{E}[n]$ that is used to determine the next edge to be explored. 

We can now provide the formal termination condition: the algorithm terminates at stage $t$, if during stage $t$, either one of the following conditions occur:
\begin{itemize}
    \item $A_t=\emptyset$ during stage $t$ or $A_{t+1}=\emptyset$ during the preprocessing step (P) at the beginning of stage $t+1$ as described below, i.e.\ the algorithm runs out of active vertices;
    \item $(A_t\cup B_t)\cap \mathcal{T} \neq \emptyset$, i.e.\ a path from $o$ to $\mathcal{T}$ in $H_n$ is discovered.
\end{itemize}

Upon termination at stage $t$, the algorithm returns $A_t, B_t$, all discovered open and closed edges (both in $H_n$ and $G$) and whether or not they are tagged. Here, an open edge refers to an edge that is present in $H_n$ or $G$.

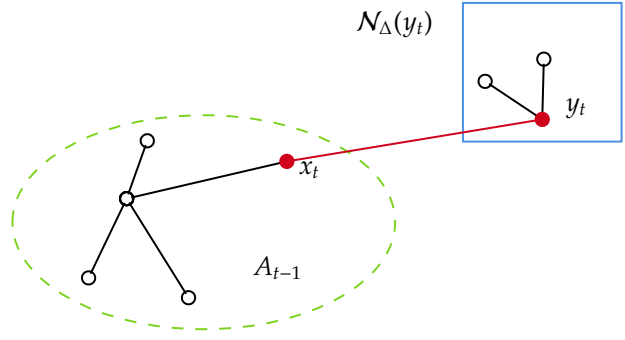
\begin{figure}
	\tikzset{every picture/.style={line width=0.75pt}}        
	\begin{center}
		\begin{tikzpicture}[x=0.75pt,y=0.75pt,yscale=-1,xscale=1]
		
			\draw  [color={rgb, 255:red, 126; green, 211; blue, 33 }  ,draw opacity=1 ][dash pattern={on 4.5pt off 4.5pt}] (12,171.5) .. controls (12,141.68) and (55.2,117.5) .. (108.5,117.5) .. controls (161.8,117.5) and (205,141.68) .. (205,171.5) .. controls (205,201.32) and (161.8,225.5) .. (108.5,225.5) .. controls (55.2,225.5) and (12,201.32) .. (12,171.5) -- cycle ;
			\draw    (71.89,159.07) -- (150.4,140.8) ;
			\draw [shift={(69.6,159.6)}, rotate = 346.9] [color={rgb, 255:red, 0; green, 0; blue, 0 }  ][line width=0.75]      (0, 0) circle [x radius= 3.35, y radius= 3.35]   ;
			\draw    (70.84,161.59) -- (99.56,207.61) ;
			\draw [shift={(100.8,209.6)}, rotate = 58.04] [color={rgb, 255:red, 0; green, 0; blue, 0 }  ][line width=0.75]      (0, 0) circle [x radius= 3.35, y radius= 3.35]   ;
			\draw [shift={(69.6,159.6)}, rotate = 58.04] [color={rgb, 255:red, 0; green, 0; blue, 0 }  ][line width=0.75]      (0, 0) circle [x radius= 3.35, y radius= 3.35]   ;
			\draw    (70.39,157.39) -- (79.21,132.61) ;
			\draw [shift={(80,130.4)}, rotate = 289.6] [color={rgb, 255:red, 0; green, 0; blue, 0 }  ][line width=0.75]      (0, 0) circle [x radius= 3.35, y radius= 3.35]   ;
			\draw [shift={(69.6,159.6)}, rotate = 289.6] [color={rgb, 255:red, 0; green, 0; blue, 0 }  ][line width=0.75]      (0, 0) circle [x radius= 3.35, y radius= 3.35]   ;
			\draw    (51.42,197.48) -- (68.58,161.72) ;
			\draw [shift={(69.6,159.6)}, rotate = 295.64] [color={rgb, 255:red, 0; green, 0; blue, 0 }  ][line width=0.75]      (0, 0) circle [x radius= 3.35, y radius= 3.35]   ;
			\draw [shift={(50.4,199.6)}, rotate = 295.64] [color={rgb, 255:red, 0; green, 0; blue, 0 }  ][line width=0.75]      (0, 0) circle [x radius= 3.35, y radius= 3.35]   ;
			\draw    (252.36,101.7) -- (279.2,119.6) ;
			\draw [shift={(250.4,100.4)}, rotate = 33.69] [color={rgb, 255:red, 0; green, 0; blue, 0 }  ][line width=0.75]      (0, 0) circle [x radius= 3.35, y radius= 3.35]   ;
			\draw    (279.2,119.6) -- (279.94,91.55) ;
			\draw [shift={(280,89.2)}, rotate = 271.51] [color={rgb, 255:red, 0; green, 0; blue, 0 }  ][line width=0.75]      (0, 0) circle [x radius= 3.35, y radius= 3.35]   ;
			\draw  [color={rgb, 255:red, 74; green, 144; blue, 226 }  ,draw opacity=1 ] (239.2,60.8) -- (319.2,60.8) -- (319.2,130.8) -- (239.2,130.8) -- cycle ;
			\draw [color={rgb, 255:red, 208; green, 2; blue, 27 }  ,draw opacity=1 ]   (150.4,140.8) -- (279.2,119.6) ;
			\draw [shift={(279.2,119.6)}, rotate = 350.65] [color={rgb, 255:red, 208; green, 2; blue, 27 }  ,draw opacity=1 ][fill={rgb, 255:red, 208; green, 2; blue, 27 }  ,fill opacity=1 ][line width=0.75]      (0, 0) circle [x radius= 3.35, y radius= 3.35]   ;
			\draw [shift={(150.4,140.8)}, rotate = 350.65] [color={rgb, 255:red, 208; green, 2; blue, 27 }  ,draw opacity=1 ][fill={rgb, 255:red, 208; green, 2; blue, 27 }  ,fill opacity=1 ][line width=0.75]      (0, 0) circle [x radius= 3.35, y radius= 3.35]   ;
			
			\draw (155.2,139.4) node [anchor=north west][inner sep=0.75pt]    {$x_{t}$};
			\draw (289.2,107.2) node [anchor=north west][inner sep=0.75pt]    {$y_{t}$};
			\draw (132.4,187.2) node [anchor=north west][inner sep=0.75pt]    {$A_{t-1}$};
			\draw (184,63.4) node [anchor=north west][inner sep=0.75pt]    {$\mathcal{N}_{\Delta }( y_{t})$};

		\end{tikzpicture}
		
	\end{center}
	\caption{Schematic depiction of part of the exploration scheme: when the active edge $x_ty_t$ (red) is explored, a passed (F)-check ensures that the only unexplored vertices that can be reached from $y_t$ are in $\mathcal{N}_{\Delta }( y_{t})$ (blue rectangle). In the coupling of exploration to random graphs, these edges correspond to edges in $G_J$ that are potentially all removed in $G_{J'}$, in which case $x_ty_t$ becomes irrelevant for the percolation event in $G_{J'}$.}
\end{figure}

We now describe an exploration stage $t\geq 1$, conditionally on the event that the algorithm has not terminated at any stage $s<t$.
\begin{itemize}
    \item [(P)] \emph{Preprocessing}: 
     \begin{enumerate}[(P.a)]
         \item If there exist vertices in $A_{t-1}$ without incident edges in $L_{t-1}$, then remove these vertices from $A_{t-1}$ and include them into $B_{t-1}$.
         \item Pick the smallest edge $e_t$ in $L_t$ that is adjacent to the set $A_{t-1}$. 
         \item  Set $L_t=L_{t-1}\setminus \{e_t\}$ and begin the exploration of $e_t$ with step (1) below.
     \end{enumerate} 
\item[(1)] Decide whether the marks of $e_t$ are revealed or not: 
\begin{enumerate}[({1}.a)] \item If both endpoints $x_t$ and $y_t$ of $e_t$ are in $A_{t-1}$, then the edge is irrelevant for the cluster exploration. Advance to the next stage $t+1$.
    \item Otherwise, relabel $e_t=x_ty_t$ with $x_t\in A_{t-1}, y_t\notin A_{t-1}$ and go to (2).
    \end{enumerate} 
\item[(2)] Check, if $(e_t,U_{e_t})\in E_{t-1}$, i.e.\ if the edge mark of $e_t$ has been revealed in a previous step.
If yes move to step (F), if no move to step (3).
\item[(3)] Reveal $U_{e_t}$.
\begin{enumerate}[({3}.a)]
    \item If $U_{e_t}>J_{e_t}$, then $e_t$ is closed in $G$. Advance to stage $t+1$.
    \item If $U_{e_t}\leq J_{e_t}$, then $e_t$ is open in $G$.
    \begin{enumerate}[({3.b}.i)]
        \item If $x_ty_t\notin \Delta$, then perform the check (F) and advance to (4).
        \item If $x_ty_t\in \Delta$, then advance to (5).
        \end{enumerate}
\end{enumerate}
\item[(4)] Proceed according to whether $x_ty_t\notin \Delta$ passed (F) or failed:
    \begin{enumerate}[({4}.a)]
    \item In case of failure, set $A_{t}=A_{t-1}\cup{y_t}$. The edge $e_t$ is open in $H_n$. Advance to stage $t+1$.
    \item Otherwise, set $e_t$ to open in $H_n$ and advance to the S-check (S).
    \end{enumerate}
\item[(5)]Reveal $W_{e_t}$.
     \begin{enumerate}[({5}.a)]
        \item If $W_{x_ty_t}\leq 1-q$, then $e_t$ is open in $H_n$. Set $A_{t}=A_{t-1}\cup{y_t}$, advance to stage $t+1$.
        \item If $W_{x_ty_t}>1-q$, then $e_t$ is open in $G$ but closed in $H_n$. Set $A_{t}=A_{t-1}$, advance to stage $t+1$.
     \end{enumerate}    
\item[(F)] \emph{F-check}. 
\begin{enumerate}[({F}.a)]
\item Set $\mathcal{N}_{\Delta^\mathsf{C}}(y_t)\coloneqq \{z\in B_\Gamma(o,N) \colon z \notin A_{t-1}, y_tz\notin \Delta, y_tz\in L_{t-1}\}$. 
\item Reveal $U_{y_tz}$ for all $z\in \mathcal{N}_{\Delta^\mathsf{C}}(y_t)$.
\item The edge $e_t$ passes (F) if 
    \[ \sum_{z\in \mathcal{N}_{\Delta^\mathsf{c}}(y_t)}\1\{U_{y_t z}\leq J_{y_t z}\}=0,\] and otherwise it fails. 
    \item Remove all edges $y_tz$ with $z\in \mathcal{N}_{\Delta^\mathsf{c}}(y_t)$ such that $U_{y_t z}> J_{y_t z}$ from $L_t$. 
    \item Set \[E_t=E_{t-1}\cup\{U_{y_tz}:y_tz\in \mathcal{N}_{\Delta^\mathsf{c}}(y_t), U_{y_tz}\leq J_{y_tz}\},\]
    and proceed with (4).
    \end{enumerate}
    \item[(S)] \emph{S-check}: 
    \begin{enumerate}[({S}.a)]
\item Set $\mathcal{N}_{\Delta}(y_t) \coloneqq \{z\in B_\Gamma(o,N) \colon z \notin A_{t-1}, y_tz\in \Delta\}$.
\item $e_t$ receives a tag if 
    \[\sum_{z\in \mathcal{N}_{\Delta}(y_t)}\1\{V_{y_t z}^{y_t}\leq 1-q\}=0.\] 
         In that case, set $B_t=B_{t-1}\cup\{y_t\}$ and remove all edges incident to $y_t$ from $L_t$. 
    \item  $e_t$ remains untagged if 
    \[\sum_{z\in \mathcal{N}_{\Delta}(y_t)}\1\{V_{y_t z}^{y_t}\leq 1-q\}>0.\]
    In that case, set $A_t=A_{t-1}\cup\{y_t\}$ and remove all edges $y_tz$ with $z\in \mathcal{N}_{\Delta}(y_t)$ such that $V_{y_t z}^{y_t}> 1-q$ from $L_t$. 
    \item Advance to stage $t+1$.
    \end{enumerate}
\end{itemize}
The exploration algorithm described above is carefully constructed so that independence is preserved from one iteration to the next. Particularly, it  satisfies the following properties. 
\begin{lemma}\label{lem:exploration}
 \begin{enumerate}[(i)]
     \item If $U_{x_ty_t}\in E_{t-1}$ in step (2), then only $x_t$ can have been involved in a previous failed F-check.      
     \item If at exploration stage $t$ an edge is tagged, i.e.\ if \[\sum_{z\in \mathcal{N}_{\Delta}(y_t)}\1\{V_{y_t z}^{y_t}\leq 1-q\}=0\] in the S-check, then the exploration has already found all possible neighbours of $y_t$ in $H_n$ prior to stage $t$.
    \item The additional edge marks revealed during the F-check at some stage $t$ cannot have been encountered at any stage $0\leq s<t.$
    \item The auxiliary edge marks revealed during the S-check at some stage $t$ cannot have been encountered at any stage $0\leq s<t.$
    \item If $(e_t,U_{e_t})\in E_{t-1}$ in step (2), then $e_t\notin \Delta$.
 \end{enumerate}   
\end{lemma}
\begin{proof}
\begin{enumerate}[(i)]
    \item Suppose $U_{x_ty_t}\in E_{t-1}$ in step (2). Then, for some $s<t$, $e_s=x_sy_s$ and an F-check was performed, involving $e_t=y_sz$ for some $z \in \mathcal{N}_{\Delta^\mathsf{c}}(y_s)$. If the F-check failed, then in Step (4), one would include $y_s$ in $A_s$. Since $y_t \notin A_t$, by the rules of the exploration algorithm, necessarily $y_s=x_t$.  
    \item Because the step (S) is only performed if $e_t=x_ty_t$ has passed (F), the only non-revealed neighbours of $y_t$ are reached through $\Delta$. Note that $H_n$-neighbours in $A_{t-1}$ and $B_{t-1}$ are already known (in fact there cannot be any in $B_{t-1}$). All possible future connections outside $A_{t-1}$ are discarded through step (S) itself.
    \item Assume the opposite. Firstly, neither $y_t$ or $z$ could have become active before, according to the rules of the algorithm. Secondly, if either $y_t$ or $z$ had been involved in a previous F-check, then this check would necessarily have failed. But then the involved endpoint $y_t$ or $z$ would have been activated during step (4), which produces the same contradiction.
    \item This is similar to (iii): the additional marks revealed in (S) could only have been encountered before, if $y_t$ had been involved in a previous S-check. But then $y_t$ would have become either activated or boundary, which is not possible.
    \item This follows from only edges in $\Delta^\mathsf{c}$ being recorded in $E_t$ during step (F).
\end{enumerate}
\end{proof}

\subsection{Coupling}\label{sec:couple} We next show that the way in which the above exploration algorithm uncovers the random graphs allows us to establish the desired coupling. Recall that if $J<J'$ there is always a \emph{canonical coupling} between $G_J$ and $G_{J'}$ such that under the coupling $G_J$ is a subgraph of $G_{J'}.$  

\begin{prop}\label{prop:coupling}
Consider $J\in \mathscr{J}(\Gamma,\mathsf{S})$ and $J'\in \mathscr{J}_{<1}(\Gamma,\mathsf{S})$ such that $J'<J$. Set $\Delta = \textup{supp}(J-J') \subset \Gamma$, assumed to be finite. 
Then there is $p=p(J,J') \in (0,1)$ and a coupling 
of $G_{J'},G_{J}$ and $G_{pJ}$ under which, for each $N \in \N$, 
\begin{equation}\label{eq:conditionP2}
o\overset{G_{pJ}}{\not\leftrightarrow} B_\Gamma(o,N)^c \text{ implies } C'_o\subset \big\{B_{G_J}(C_o,1)\big\},
\end{equation}
where $C_o$ denotes the cluster containing $o$ in $G_{pJ}$, $C'_o$ denotes the cluster of $o$ in $G_{J'}$ and $B_{G_J}(C_o,1)$ denotes the subgraph of $G_J$ obtained from $C_o$ viewed as a subgraph of $G_J$ under the canonical coupling together with all edges emanating from $C_o$ in $G_J$.
\end{prop}

\begin{proof}
Firstly, we claim that there is $p\in (0,1)$ satisfying
\begin{equation}\label{eq:conditionP}1-p \leq\left( \min_{e\in\Delta}\Big\{1-\sqrt[3]{\tfrac{J'_e}{J_e}}\Big\} \right) \wedge \left( \min_{e\in\Delta}\Big\{1-\sqrt[3]{\tfrac{J'_e}{J_e}}\Big\}^{\sharp \Delta}\prod_{z\in\Delta^\mathsf{c}}(1-J_{oz}) \right).
\end{equation}
Indeed, we have that 
\begin{equation}\label{eq:pbound1}
\prod_{z\in\Delta^\mathsf{c}}(1-J_{oz}) = e^{\sum_{z\in\Delta^\mathsf{c}} \ln(1-J_{oz})} = e^{-\sum_{z\in\Delta^\mathsf{c}} \sum_{n\geq 1} J_{oz}^n/n}.
 \end{equation}
Since $a = \max_{z \in \Delta^\mathsf{c}} J_{oz}<1$ and $\sum_{z\in\Delta^\mathsf{c}} J_{oz}<\infty$ we have
\[\sum_{z\in\Delta^\mathsf{c}} \sum_{n\geq 1} J_{oz}^n/n = \sum_{n\geq 1} \sum_{z\in\Delta^\mathsf{c}} J_{oz}^n/n \leq \sum_{n\geq 1} \sum_{z\in\Delta^\mathsf{c}} a^n  \frac{J_{oz}}{a}  \leq \Big(\sum_{z\in\Delta^\mathsf{c}} J_{oz} \Big) \frac{1}{1-a} < \infty.\]
Inserting this bound into \eqref{eq:pbound1} yields $\prod_{z\in\Delta^\mathsf{c}}(1-J_{oz})>0$ so that both terms on the righthand side of \eqref{eq:conditionP} are strictly positive.

Now, to establish the coupling, colour all edges in $\Gamma^{[2]}$ independently either red with probability $1-p$ or black with probability $p$. Clearly, the black cluster containing $o$ in $G_J$ can be viewed as a realisation of the cluster of $o$ in $G_{pJ}$. Let $R=(R_e)_{e\in \Gamma^{[2]}}$ denote the indicator field of the red edges. 
We use the exploration algorithm with $\Delta$ as above and $n=N$, to couple $R$ with the exploration run to uncover $H_N\subset G_J$ in such a way that every tagged edge and every edge found closed in $H_N$ during step (5) is red. For this, let $\mathcal{F}_{t-1}$ denote the filtration generated by the exploration process up to stage $t-1$. 
Furthermore, $H_N$ is constructed such that it is a spanning tree containing a spanning tree of $C'_o$ in the case that the exploration terminates before $\mathcal{T}(G,N)$ is reached. 
For this, we set $q$ equal to the righthand side of \eqref{eq:conditionP}. 
Then the percolation cluster obtain by declaring an edge $e=xy$ open if and only if $U_e<J_e$, $V_{x,y}^x < 1-q$, $W_{xy}<1-q$ and $V_{xy}^y <1-q$ stochastically dominates that of  $C_o'$. Indeed, since these random variables are all independent, for any edge $e=xy$ it holds that
\[ \P \left( U_e<J_e, V_{x,y}^x < 1-q, W_{xy}<1-q, V_{xy}^y <1-q \right) =J_e (1-q)^3 \geq J_{e}'.\]
Particularly, it suffices to show that 
\[ \P(e_t \text{ tagged or closed in }H_N|\mathcal{F}_{t-1})\geq 1-p.\] 
Clearly, if $e_t\in \Delta$ then this is implied by
\[q\geq 1-p,\]
since we may couple $W_{e_t}$ and the independent colouring Bernoulli $R_{e_t}$. It remains to analyse the probability that $e_t$ becomes tagged. For this we note that $e_t$ needs to first pass (F). Given that $e_t$ is present in $U_{e_t}\leq J_{e_t}$, this has conditional probability
\[
\prod_{z\in \mathcal{N}_{\Delta^\mathsf{C}}(y_t)}(1-J_{y_tz})\geq \prod_{z\in\Delta^c}(1-J_{oz}).
\]
Conditionally on the passed (F) check, the tagging probability is
\[
q^{\sharp  \mathcal{N}_{\Delta}(y_t)}\geq q^{\sharp \Delta},
\]
hence the overall probability is at least
\[
q^{\sharp \Delta}\prod_{z\in\Delta^c}(1-J_{oz})\geq 1-p
\]
by construction. Now note that the coupling with $R_{e_t}$ can be achieved, since the involved random edge marks are independent of $\mathcal{F}_{t-1}$ by Lemma~\ref{lem:exploration}: if both $x_t$ and $y_t$ have not been involved in the exploration at any prior stage, then this means that none of the edges in $\N_\Delta(y_t)$ have been previously encountered, hence their occupation status and marks are independent of $\mathcal{F}_{t-1}$. If $e_t$ has been encountered before, then by Lemma~\ref{lem:exploration}(i), this only revealed knowledge about edges adjacent to $x_t$ in $G$. By Lemma~\ref{lem:exploration}(ii), only potential edges $y_tz$ to vertices $z$ for which no path from $o$ in $H_N$ has been uncovered yet are relevant for the tagging, which implies that the corresponding marks revealed in the exploration are independent of $\mathcal{F}_{t-1}$.

If the exploration terminates, before $\mathcal{T}(G,N)$ is reached, then $H_N$ is a spanning tree for $C_o'$ under the coupling. By definition, the tagged edges must end in leaves of $H_N$. Since $C_o$ dominates the untagged part of $H_N$ in the coupling, the assertion
\[
C'_o\subset \big\{B_{G_J}(C_o,1)\big\}
\]
follows.
\end{proof}

\subsection{Derivation of main results}\label{sec:proofs_fin}
We first deduce Theorem~\ref{thm:main2}, since it is used in the other proofs. 
For this, we 
make use of \cite[Theorem 1.1]{duminil-copin_new_2016}. Note that in \cite{duminil-copin_new_2016} a version of long-range percolation is used where the connectivity function is of the form $1-\exp(-\beta \phi(y-x))$ for some $\phi:\Z^d\to[0,\infty]$ and $\beta>0$. It is not difficult to see, that the proofs of \cite{duminil-copin_new_2016} apply in our setup as well. Conversely, as we discuss in detail at the beginning of Section~\ref{sec:conclusion}, our proofs can be adapted to the model of \cite{duminil-copin_new_2016}.

\begin{proof}[Proof of Theorem~\ref{thm:main2}]
 Let $p\in(0,1)$ as in Proposition~\ref{prop:coupling}. Since $J''$ is critical, $G_{pJ''}$ is subcritical by definition. Therefore, as follows by \cite[Theorem 1.1.]{duminil-copin_new_2016}, parts (2),  the cluster under $G_{pJ''}$ has finite susceptibility in the sense that
 \[
	\E\left[\left|\left\{x\in \Gamma: o\overset{G_{pJ''}}{\leftrightarrow}x\right\}\right|\right]<\infty.
	\]
From this, by Proposition \ref{prop:coupling} and since $J$ is summable,  it follows that also the cluster of $G_J$ has finite susceptibility. 
    Consequently, again by \cite[Theorem 1.1.]{duminil-copin_new_2016}, parts (2), there is an $\epsilon>0$ such that also $G_{(1+\epsilon)J}$ has finite susceptibility. Thus, $J$ is subcritical. Further, if $J$ in addition is finitely supported, then the sharpness result \cite[Theorem 1.1.]{duminil-copin_new_2016}, parts (3), apply.
\end{proof}

Utilising Theorem \ref{thm:main2}, we  next prove Theorem \ref{thm:main} and then Theorem \ref{thm:main3}.

\begin{proof}[Proof of Theorem \ref{thm:main}]
It is clear from the definitions that strong criticality implies criticality. Now assume that $J$ is critical. Firstly, let $J'<J$. Then the finiteness of the expected size of the origin cluster in $G_{J'}$ is immediate from Theorem~\ref{thm:main2}. Secondly, assume for contradiction that for $J''>J$ percolation does not occur. Without loss of generality, we may assume that $J''\in \mathscr{J}_{<1}(\Gamma,\mathsf{S})$. Note that $J''$ cannot be supercritical, since we assumed that percolation does not occur. Furthermore, $J''$ cannot be subcritical, since that would contradict criticality of $J$. Hence $J''$ must be critical. But since $J''\in \mathscr{J}_{<1}(\Gamma,\mathsf{S})$ and $J<J''$, Theorem~\ref{thm:main2} implies that $J$ is subcritical, which is a contradiction. Hence percolation must occur for $J''$.
\end{proof}

\begin{proof}[Proof of Theorem \ref{thm:main3}]
If $J'$ is critical, then $J>J'$ cannot be critical as well, since that would contradict Theorem~\ref{thm:main2}. By monotonicity, $J$ cannot be sub-critical either and we conclude that $J$ must be supercritical.
\end{proof}

\section{Extensions and related recent work.}\label{sec:conclusion}

\paragraph{Alternative representation of connection probabilities.}

We may express long-range percolation via a connectivity function $J$ of the form
\[
J_e=1-\textup{e}^{-\varphi_e}, \quad e\in\Gamma^{[2]},
\]
where $\varphi:\Gamma^{[2]}\to[0,\infty].$ This representation is used, for instance, in \cite{duminil-copin_new_2016} and \cite{bäumler2025continuitycriticalvalueshape}.  Note that our definitions for super- and subcriticality are then unnatural and should be replaced by corresponding domination statements for $\varphi$ instead of $J$. 
More specifically, long-range percolation on $\Z^d$ is usually studied for $J^{\beta\phi}_{xy}=1-\exp(-\beta \phi(y-x))$, where $\phi:\Z^d\to[0,\infty]$ is such that $\sum_{z:|z|\geq \ell} \phi(z)<\infty$ for some $\ell \in \N$, and $\beta>0$ serving as an edge density parameter. Given $\phi$, one may define \[\beta_\mathsf{c} \coloneqq \inf\big\{\beta>0: G_{J^{\beta\phi}} \text{ percolates} \big\}\in[0,\infty],\]
and associate the notions of criticality, subcriticality, and supercriticality with the regimes $\beta=\beta_{\mathsf{c}}$, $\beta<\beta_{\mathsf{c}}$, and $\beta>\beta_\mathsf{c},$ respectively.

There are several ways to see that our results remain true in this setting. One argument goes as follows: first we restrict ourselves to the case where $\varphi<\infty$ everywhere. Then $G_J=G_\varphi$ can be dominated in the obvious way by a \emph{multigraph} $\bar{G}_\varphi$ in which for each potential edge $e$ an independent Poisson($\varphi_e$) distributed number of edges are placed. Now observe that independent bond percolation with retention parameter $p$ in this multigraph can then be coupled to the model $G_{p\varphi}$ by Poisson thinning. In particular, our proofs apply under the convention that parallel edges are always explored, checked and tagged simultaneously, since the corresponding correlated thinning is easily seen to always take more edges away than the independent edge thinning with probability $q$ equal to the righthand side of \eqref{eq:conditionP}. If we allow $\varphi=\infty$, then, just as in Theorem~\ref{thm:main}, the additional qualification that the downward perturbation $\varphi'$ is finite everywhere applies.

\paragraph{Long-range percolation on $\Z^d$.}

Proposition 1.10 of \cite{bäumler2025continuitycriticalvalueshape} is the only previous result about strict inequality of critical points in long-range lattice percolation that we are aware of.This previous result is limited to connection functions $J$ on $\Gamma=\Z^d$ that preserve all lattice symmetries, and, more importantly, the use of differential inequalities in its derivation requires that $J$ satisfy the Lipschitz-type condition
\[
0<aJ(z+e)\leq J(z)\leq AJ(z+e)
\]
for some global constants $0<a,A<\infty$ and any nearest neighbour $e$ of $0\in\Z^d$. Theorem~\ref{thm:main3} therefore implies a strengthening of \cite[Theorem 1.9]{bäumler2025continuitycriticalvalueshape}, which transfers results between different notions of supercriticality in the Euclidean setting  $\Gamma=\Z^d$. 
In particular, $\beta>\beta_\mathsf{c}$ is equivalent to supercriticality of $J^{\beta\phi}$ as defined in the previous paragraph. Theorem~\ref{thm:main3} thus enables us to extend many properties of supercritical clusters to the model with connection function 
\[
\tilde{J}=1-\textup{e}^{-\beta_{\mathsf{c}}\phi+f},
\]
where $f>0$ is not necessarily a multiple of $\phi$.  
For instance, this includes, under additional regularity assumptions on $\tilde J$ in each case,
\begin{enumerate}[(i)]
	\item the truncation property, and related continuity and approximation results \cite{Berger2002,bäumler2025continuitycriticalvalueshape};
	\item upper bounds on typical distances \cite{Biskup04,BiskupLin19}, and diameter \cite{Biskup_diam11} in the infinite cluster;
	\item a shape theorem \cite{bäumler2025continuitycriticalvalueshape};
	\item transience of the infinite cluster \cite{Berger2002,bäumler2025continuitycriticalvalueshape}.
\end{enumerate}

\paragraph{Application to nearest neighbour models.}

Let $J^{(d)}$ denote the connectivity function on potential edges of $\Gamma=\Z^d$ that assumes the value $1$ if evaluated at a nearest neighbour edge and $0$ in all other cases. If $d'<d$, then we may view $J^{d'}$ as a connectivity function on $\Z^d$ that is supported on a sub-lattice. In particular, we have $pJ^{d'}<pJ^{d}$ for all $p\in(0,1]$, and therefore obtain the original result of Aizenman and Grimmett \cite{AizGrim91,balister2014essentialenhancementsrevisited} on nearest neighbour percolation as a special case of Theorem~\ref{thm:main2}, namely that the critical probability $p_\mathsf{c}(\Z^d)$ decreases strictly with the dimension $d\geq 1$.

This result holds in much greater generality. For instance, \cite{Martineau19} concluded strict inequalities general covering maps, for both site and bond percolation, under rather mild conditions; see Theorem 2.1 therein.  
Whilst the present paper was written, the work \cite{martineau2025stochasticdominationliftsrandom} appeared,  
providing such strict inequalities via an alternative approach for bond percolation, see Theorem 5.1 therein. 
Interestingly, the approach in \cite{martineau2025stochasticdominationliftsrandom} is based on coupling methods of a similar flavour as our argument, and not differential equations as in \cite{AizGrim91,Martineau19}. 
Nevertheless, the proof in  \cite{martineau2025stochasticdominationliftsrandom}  is different from ours in that they study enhancements, i.e.\ the effect of inserting additional edges to a graph, instead of removing edges as in our Proposition \ref{prop:coupling}. 
Moreover, and in contrast to \cite{martineau2025stochasticdominationliftsrandom}, our main interest is in strict inequalities upon local perturbation of the percolation parameter when the graph $G$ is fixed. 
Therefore, in the same vein as discussed above for the case of $G=\Z^d$, by setting $J'_{e} =0$ for $e\in \Delta$, we in fact directly cover some of these results in the setting of transitive graphs. 

\paragraph{Directed and oriented percolation.} 

The standard essential enhancement method fails for directed and oriented percolation models, as noted in \cite{LimaUngarettiVares2024,Archer2025,Terra2025}. 
However, partly relying on the use of coupling arguments somewhat akin to ours, there has recently been progress on strict inequalities for such models too. 
Particularly, for $G$ a connected graph with bounded degree,  \cite[Theorem 4]{LimaUngarettiVares2024} states that the critical value of directed bond percolation on $G$ is strictly larger than the corresponding one on the so-called ladder graph  $G\times \Z_+$. From this they also concluded that the critical probability for oriented percolation on $\Z^d$ decreases strictly with the dimension, see \cite[Corollary 3]{LimaUngarettiVares2024}. See also \cite[Theorem 5]{LimaUngarettiVares2024}  for a statement involving directed site percolation and \cite{Terra2025} for a related result for oriented bond percolation on $\Z^2$.

We extend and generalize these advances to the setting of long-range bond directed percolation models on a locally finite vertex transitive graph $G$. Indeed, upon minor modifications, the exploration algorithm provided in Section \ref{sec:proof} also applies to this case. That is, consider the independent Bernoulli configuration $\omega$ on $\Gamma^{[2]}$ satisfying
\[ \P\left(\omega_{(x,y)}=1\right) = J_{(x,y)}, \quad x,y \in \Gamma \text{ with }x\neq y,\]
and where now $J_{(x,y)}$ is not necessarily equal to $J_{(y,x)}$. 
This gives a random directed subgraph $G(\omega)$ of $(\Gamma,\Gamma^{[2]})$. 
Then, writing  $\smash{\{ o \overset{G_{J}}{\rightarrow}x\}}$ for  the event $o$ is connected to $x$ using the directed edges of the corresponding random directed subgraph $G_J$, the statement of Proposition \ref{prop:coupling} extends to this setting after replacing $\smash{\{ o \overset{G_{J}}{\leftrightarrow}x\}}$  by $\smash{\{ o \overset{G_{J}}{\rightarrow}x\}}$ wherever relevant. Moreover, as noted \cite[Section 1.2]{duminil-copin_new_2016}, their sharpness results applies also to directed percolation. Therefore, our main results as described in Section \ref{sec:percolation} transfer, yielding e.g.\ strict inequalities at criticality. 
This covers for instance the special case of oriented percolation, or the discrete-time contact process, where $\Gamma=\Gamma'\times \Z$ with $\Gamma'$ a locally finite vertex transitive graph and $J_{(x,n),(y,m)}>0$ may be non-zero only if $m=n+1$, and thereby yield a vast generalisation of \cite[Corollary 3]{LimaUngarettiVares2024} and \cite[Theorem 1]{Terra2025}. 

\paragraph{Beyond transitive graphs. }

The setting of this paper is that of invariant independent percolation on transitive graphs, following \cite{duminil-copin_new_2016}. This covers the main case of interest for us, which is long-range percolation on $\Z^d$ as discussed in the introduction and above, and furthermore a variety of important non-Euclidean examples without adding too much technical overhead.  

We believe that our methods can be adjusted to work in more general settings too, such as on quasi-transitive graphs and for percolation processes that are not completely independent.  For this, note that Proposition~\ref{prop:coupling} in principle still holds if we instead only have a uniform bound of the kind that
	\[\inf_t \P(e_t \text{ tagged in }H_N|\mathcal{F}_{t-1})>0.\]
In particular, the sensitivity of criticality remains valid under this assumptions; only the `sharpness'-type statements that involve further properties of the sub- and supercritical phases require transitivity in as far as they rely on previous work for transitive graphs. 
For instance, the sharpness-results that we apply are known to hold for quasi-transitive graphs, as concluded in \cite[Section 8]{Antunovic2008}.  
However, it seems to us that the present setting illustrates the approach in the clearest possible way and we leave further adaptations to future work.

Similarly, our approach can presumably  also be adapted to continuum percolation models on Poisson processes over homogeneous spaces, at least as long as edges remain independent. In this setting,  sharpness-results for continuum percolation with unbounded range were recently obtained in \cite{Higgs2025}. However, the coupling needed involves a combination of bond and site percolation and is a little more elaborate than the one devised for the present paper.

\noindent\textbf{\large Funding acknowledgement.} CM's research was partially funded by the Deutsche Forschungsgemeinschaft Deutsche Forschungsgemeinschaft (DFG, German Research Foundation) – SPP 2265 443916008.

\section*{References}
\renewcommand*{\bibfont}{\footnotesize}
\printbibliography[heading = none]

@article{AB87,
 author = {Aizenman, Michael and Barsky, David J.},
 title = {Sharpness of the phase transition in percolation models},
 fjournal = {Communications in Mathematical Physics},
 journal = {Commun. Math. Phys.},
 issn = {0010-3616},
 volume = {108},
 pages = {489--526},
 year = {1987},
 language = {English},
 doi = {10.1007/BF01212322},
 zbMATH = {4003223},
 Zbl = {0618.60098}
}

@misc{Higgs2025,
      title={Exponential decay for the random connection model using asymptotic transitivity}, 
      author={Frankie Higgs},
      year={2025},
      eprint={2509.02310},
      archivePrefix={arXiv},
      primaryClass={math.PR},
      url={https://www.arxiv.org/abs/2509.02310}, 
}

@article {Antunovic2008,
    AUTHOR = {Antunovi\'c, Ton\'ci and Veseli\'c, Ivan},
     TITLE = {Sharpness of the phase transition and exponential decay of the
              subcritical cluster size for percolation and quasi-transitive
              graphs},
   JOURNAL = {J. Stat. Phys.},
  FJOURNAL = {Journal of Statistical Physics},
    VOLUME = {130},
      YEAR = {2008},
    NUMBER = {5},
     PAGES = {983--1009},
      ISSN = {0022-4715,1572-9613},
   MRCLASS = {82B43 (82B26)},
  MRNUMBER = {2384072},
MRREVIEWER = {Johan\ Tykesson},
       DOI = {10.1007/s10955-007-9459-x},
       URL = {https://doi.org/10.1007/s10955-007-9459-x},
}

@article{Archer2025,
	author = {Archer, Eleanor and Hartarsky, Ivailo and Kolesnik, Brett and Olesker-Taylor, Sam and Schapira, Bruno and Valesin, Daniel},
	doi = {10.1007/s00440-025-01406-4},
	issn = {1432-2064},
journal ={Probab. Theory Related Fields},	
fjournal = {Probability Theory and Related Fields},
	title = {Catalan percolation},
	url = {https://doi.org/10.1007/s00440-025-01406-4},
	year = {2025},
}

@article{Terra2025,
 author = {Terra, C{\'e}lio},
 title = {Monotonicity of critical point in two-dimensional oriented percolation with enhancement},
 fjournal = {Brazilian Journal of Probability and Statistics},
 journal = {Braz. J. Probab. Stat.},
 issn = {0103-0752},
 volume = {39},
 number = {2},
 pages = {204--209},
 year = {2025},
 language = {English},
 doi = {10.1214/25-BJPS631},
 zbMATH = {8102859}
}

@misc{bäumler2025localcriteriaglobalconnectivity,
      title={Local criteria for global connectivity comparisons: beyond stochastic domination}, 
      author={Johannes Bäumler and Benedikt Jahnel and Jonas Köppl and Bas Lodewijks and Lily Reeves and András Tóbiás},
      year={2025},
      eprint={2510.03934},
      archivePrefix={arXiv},
      primaryClass={math.PR},
      url={https://arxiv.org/abs/2510.03934}, 
}

@article{Biskup04,
	author = {Biskup, Marek},
	title = {On the scaling of the chemical distance in long-range percolation models},
	fjournal = {The Annals of Probability},
	journal = {Ann. Probab.},
	issn = {0091-1798},
	volume = {32},
	number = {4},
	pages = {2938--2977},
	year = {2004},
	language = {English},
	doi = {10.1214/009117904000000577},
	zbMATH = {2148681},
	Zbl = {1072.60084}
}

@article{Biskup_diam11,
	author = {Biskup, Marek},
	title = {Graph diameter in long-range percolation},
	fjournal = {Random Structures \& Algorithms},
	journal = {Random Struct. Algorithms},
	issn = {1042-9832},
	volume = {39},
	number = {2},
	pages = {210--227},
	year = {2011},
	language = {English},
	doi = {10.1002/rsa.20349},
	zbMATH = {5963854},
	Zbl = {1228.05134}
}

@article{BiskupLin19,
	author = {Biskup, Marek and Lin, Jeffrey},
	title = {Sharp asymptotic for the chemical distance in long-range percolation},
	fjournal = {Random Structures \& Algorithms},
	journal = {Random Struct. Algorithms},
	issn = {1042-9832},
	volume = {55},
	number = {3},
	pages = {560--583},
	year = {2019},
	language = {English},
	doi = {10.1002/rsa.20849},
	zbMATH = {7138327},
	Zbl = {1428.05164}
}

@misc {LimaUngarettiVares2024,
    AUTHOR = {de Lima, Bernardo N. B. and Ungaretti, Daniel and Vares, Maria
              Eul\'alia},
     TITLE = {A note on oriented percolation with inhomogeneities and strict
              inequalities},
   JOURNAL = {Stochastic Process. Appl.},
  FJOURNAL = {Stochastic Processes and their Applications},
    VOLUME = {175},
      YEAR = {2024},
     PAGES = {Paper No. 104387, 7},
      ISSN = {0304-4149,1879-209X},
   MRCLASS = {60K35 (82B43)},
  MRNUMBER = {4758364},
       DOI = {10.1016/j.spa.2024.104387},
       URL = {https://doi.org/10.1016/j.spa.2024.104387},
}

@article {Berger2002,
    AUTHOR = {Berger, Noam},
     TITLE = {Transience, recurrence and critical behavior for long-range
              percolation},
   JOURNAL = {Comm. Math. Phys.},
  FJOURNAL = {Communications in Mathematical Physics},
    VOLUME = {226},
      YEAR = {2002},
    NUMBER = {3},
     PAGES = {531--558},
      ISSN = {0010-3616,1432-0916},
   MRCLASS = {82B43 (60J45 60K35)},
  MRNUMBER = {1896880},
MRREVIEWER = {Olle\ H\"aggstr\"om},
       DOI = {10.1007/s002200200617},
       URL = {https://doi.org/10.1007/s002200200617},
}

@article {Menshikov87,
    AUTHOR = {Menshikov, Michail V.},
     TITLE = {Quantitative estimates and strong inequalities for the
              critical points of a graph and its subgraph},
   JOURNAL = {Teor. Veroyatnost. i Primenen.},
  FJOURNAL = {Akademiya Nauk SSSR. Teoriya Veroyatnoste\u i\ i ee
              Primeneniya},
    VOLUME = {32},
      YEAR = {1987},
    NUMBER = {3},
     PAGES = {599--602},
      ISSN = {0040-361X},
   MRCLASS = {60K35 (05C35 60C05)},
  MRNUMBER = {914957},
MRREVIEWER = {B.\ M.\ Gurevich},
}

@book {KestenBook,
    AUTHOR = {Kesten, Harry},
     TITLE = {Percolation theory for mathematicians},
    SERIES = {Progress in Probability and Statistics},
    VOLUME = {2},
 PUBLISHER = {Birkh\"auser, Boston, MA},
      YEAR = {1982},
     PAGES = {iv+423},
      ISBN = {3-7643-3107-0},
   MRCLASS = {60K35 (05C70 82A43)},
  MRNUMBER = {692943},
MRREVIEWER = {R.\ T.\ Smythe},
}

@article {LSS97,
    AUTHOR = {Liggett, Thomas M. and Schonmann, Roberto H. and Stacey, Alan M.},
     TITLE = {Domination by product measures},
   JOURNAL = {Ann. Probab.},
  FJOURNAL = {The Annals of Probability},
    VOLUME = {25},
      YEAR = {1997},
    NUMBER = {1},
     PAGES = {71--95},
      ISSN = {0091-1798,2168-894X},
   MRCLASS = {60G60 (60G10 60K35)},
  MRNUMBER = {1428500},
MRREVIEWER = {Vincent\ De Valk},
       DOI = {10.1214/aop/1024404279},
       URL = {https://doi.org/10.1214/aop/1024404279},
}

@article {Mühlbacher21,
    AUTHOR = {M\"uhlbacher, Peter},
     TITLE = {Critical parameters for loop and {B}ernoulli percolation},
   JOURNAL = {ALEA Lat. Am. J. Probab. Math. Stat.},
  FJOURNAL = {ALEA. Latin American Journal of Probability and Mathematical
              Statistics},
    VOLUME = {18},
      YEAR = {2021},
    NUMBER = {1},
     PAGES = {289--308},
      ISSN = {1980-0436},
   MRCLASS = {60K35},
  MRNUMBER = {4213859},
MRREVIEWER = {Subhajit\ Goswami},
       DOI = {10.30757/alea.v18-13},
       URL = {https://doi.org/10.30757/alea.v18-13},
}

@article {AizGrim91,
    AUTHOR = {Aizenman, Michael and Grimmett, Geoffrey},
     TITLE = {Strict monotonicity for critical points in percolation and
              ferromagnetic models},
   JOURNAL = {J. Statist. Phys.},
  FJOURNAL = {Journal of Statistical Physics},
    VOLUME = {63},
      YEAR = {1991},
    NUMBER = {5-6},
     PAGES = {817--835},
      ISSN = {0022-4715,1572-9613},
   MRCLASS = {82B43 (82B27)},
  MRNUMBER = {1116036},
MRREVIEWER = {J.\ Theodore\ Cox},
       DOI = {10.1007/BF01029985},
       URL = {https://doi.org/10.1007/BF01029985},
}

@misc{balister2014essentialenhancementsrevisited,
      title={Essential enhancements revisited}, 
      author={Paul Balister and Béla Bollobás and Oliver Riordan},
      year={2014},
      eprint={1402.0834},
      archivePrefix={arXiv},
      primaryClass={math.PR},
      url={https://arxiv.org/abs/1402.0834}, 
}

@misc{bäumler2025continuitycriticalvalueshape,
      title={Continuity of the critical value and a shape theorem for long-range percolation}, 
      author={Johannes Bäumler},
      year={2025},
      eprint={2312.04099},
      archivePrefix={arXiv},
      primaryClass={math.PR},
      url={https://arxiv.org/abs/2312.04099}, 
}

@article {PenRoso11,
    AUTHOR = {Franceschetti, Massimo and Penrose, Mathew D. and Rosoman,
              Tom},
     TITLE = {Strict inequalities of critical values in continuum
              percolation},
   JOURNAL = {J. Stat. Phys.},
  FJOURNAL = {Journal of Statistical Physics},
    VOLUME = {142},
      YEAR = {2011},
    NUMBER = {3},
     PAGES = {460--486},
      ISSN = {0022-4715,1572-9613},
   MRCLASS = {60K35},
  MRNUMBER = {2771041},
MRREVIEWER = {Mu\ Fa\ Chen},
       DOI = {10.1007/s10955-011-0122-1},
       URL = {https://doi.org/10.1007/s10955-011-0122-1},
}

@phdthesis{rosoman2011critical,
  title={Critical values in continuum and dependent percolation},
  author={Rosoman, Thomas E},
  year={2011},
  school={University of Bath}
}

@article {taggi23,
    AUTHOR = {Taggi, Lorenzo},
     TITLE = {Essential enhancements in abelian networks: continuity and
              uniform strict monotonicity},
   JOURNAL = {Ann. Probab.},
  FJOURNAL = {The Annals of Probability},
    VOLUME = {51},
      YEAR = {2023},
    NUMBER = {6},
     PAGES = {2243--2264},
      ISSN = {0091-1798,2168-894X},
   MRCLASS = {82C22 (60K35 82C26)},
  MRNUMBER = {4666295},
MRREVIEWER = {Bruno\ Schapira},
       DOI = {10.1214/23-aop1647},
       URL = {https://doi.org/10.1214/23-aop1647},
}

@misc{martineau2025stochasticdominationliftsrandom,
      title={Stochastic domination and lifts of random variables in percolation theory}, 
      author={Sébastien Martineau and Rémy Poudevigne and Paul Rax},
      year={2025},
      eprint={2504.02427},
      archivePrefix={arXiv},
      primaryClass={math.PR},
      url={https://arxiv.org/abs/2504.02427}, 
}

@article{Martineau19,
 author = {Martineau, S{\'e}bastien and Severo, Franco},
 title = {Strict monotonicity of percolation thresholds under covering maps},
 fjournal = {The Annals of Probability},
 journal = {Ann. Probab.},
 issn = {0091-1798},
 volume = {47},
 number = {6},
 pages = {4116--4136},
 year = {2019},
 language = {English},
 doi = {10.1214/19-AOP1355},
 zbMATH = {7212179},
 Zbl = {1469.60338}
}

@misc{klippel25,
      title={Loop vs. Bernoulli percolation on trees: strict inequality of critical values}, 
      author={Andreas Klippel and Benjamin Lees and Christian Mönch},
      year={2025},
      eprint={2503.03319},
      archivePrefix={arXiv},
      primaryClass={math.PR},
      url={https://arxiv.org/abs/2503.03319}, 
}

@article{vanneuville2024exponential,
 author = {Vanneuville, Hugo},
 title = {Exponential decay of the volume for {Bernoulli} percolation: a proof via stochastic comparison},
 fjournal = {Annales Henri Lebesgue},
 journal = {Ann. Henri Lebesgue},
 issn = {2644-9463},
 volume = {8},
 pages = {101--112},
 year = {2025},
 language = {English},
 doi = {10.5802/ahl.230},
 zbMATH = {8068742}
}

@misc{vanneuville2023sharpness,
      title={Sharpness of Bernoulli percolation via couplings}, 
      author={Hugo Vanneuville},
      year={2023},
      eprint={2201.08223},
      archivePrefix={arXiv},
      primaryClass={math.PR},
      url={https://arxiv.org/abs/2201.08223}, 
}

@article{duminil-copin_new_2016,
	title = {A new proof of the sharpness of the phase transition for Bernoulli percolation and the Ising model},
	volume = {343},
	issn = {0010-3616,1432-0916},
	doi = {10.1007/s00220-015-2480-z},
	pages = {725--745},
	number = {2},
	journaltitle = {Communications in Mathematical Physics},
	shortjournal = {Comm. Math. Phys.},
	author = {Duminil-Copin, Hugo and Tassion, Vincent},
	date = {2016},
	mrnumber = {3477351},
}

\end{document}